\documentclass[11pt]{amsart}
\usepackage[latin1]{inputenc}
\usepackage{amsmath,amsthm,amsfonts,amssymb,amscd}
\usepackage{verbatim}
\usepackage{mathrsfs}
\usepackage{multicol}
\usepackage{tabularx}
\usepackage[usenames,dvipsnames]{color}
\usepackage{enumerate}
\usepackage{graphicx}
\usepackage{color,xcolor}
\usepackage{bbm}
\usepackage[all]{xy}
\usepackage{hyperref}

\theoremstyle{definition}
\newtheorem{thm}{Theorem}[subsection]
\newtheorem{prop}[thm]{Proposition}
\newtheorem{cor}[thm]{Corollary}
\newtheorem{con}[thm]{Conjecture}
\newtheorem{lem}[thm]{Lemma}

\newtheorem{ex}[thm]{Example}
\newtheorem{rem}[thm]{Remark}

\numberwithin{equation}{subsection}

\def\lie#1{\mathfrak{#1}}

\def\het{{\rm ht}}
\def\ev{{\rm ev}}
\def\supp{{\rm supp}}
\def\gr{{\rm gr}}

\def\soc{{\rm soc}}
\def\otm_#1^#2{\text{\scriptsize$\bigotimes\limits_{\text{\footnotesize$#1$}}^{\text{\footnotesize$#2$}}$}}

\def\qbinom#1#2{\begin{bmatrix} #1\\#2\end{bmatrix}}
\def\tqbinom#1#2{\text{$\left[\begin{smallmatrix} #1\\#2\end{smallmatrix}\right]$}}
\def\bs#1{\boldsymbol{#1}}

\def\bsr#1{\boldsymbol{\rm{#1}}}
\def\bsru#1{\underline{\boldsymbol{\rm{#1}}}}
\def\endd{\hfill$\diamond$}

\begin{document}
\title[On Truncated Weyl Modules]{On Truncated Weyl Modules}
\author[G. Fourier, V. Martins, A. Moura]{Ghislain Fourier, Victor Martins, Adriano Moura}
\thanks{}

\address{Institut für Algebra, Zahlentheorie und Diskrete Mathematik, 
Leibniz Universität Hannover, Germany}
\email{fourier@math.uni-hannover.de}
%\thanks{}

\address{Departamento de Matemática Pura e Aplicada, CCENS, Universidade Federal do Espírito Santo, Alegre - ES - Brazil, 25900-000}
\email{victor.n.martins@ufes.br}
\thanks{Part of this work was developed while the second author was a visiting Ph.D. student at the University of Cologne. He thanks Peter Littelmann and Deniz Kus for their guidance and support during that period and the University of Cologne for hospitality. He also thanks CAPES and CNPq (SWE 203324/2014-5)  for the financial support.}

\address{Departamento de Matemática, Universidade Estadual de Campinas, Campinas - SP - Brazil, 13083-859.}
\email{aamoura@ime.unicamp.br}
\thanks{A. M. was partially supported by CNPq grant 304477/2014-1 and Fapesp grant 2014/09310-5.}

\begin{abstract}
We study structural properties of truncated Weyl modules. A truncated Weyl module $W_N(\lambda)$ is a local Weyl module for $\lie g[t]_N = \lie g \otimes \frac{\mathbb C[t]}{t^N\mathbb C[t]}$, where $\lie g$ is a finite-dimensional simple Lie algebra. It has been conjectured that, if $N$ is sufficiently small with respect to $\lambda$, the truncated Weyl module is isomorphic to a fusion product of certain irreducible modules. Our main result proves this conjecture when $\lambda$ is a multiple of certain fundamental weights, including all minuscule ones for simply laced $\lie g$. We also take a further step towards proving the conjecture for all multiples of fundamental weights by proving that the corresponding truncated Weyl module is isomorphic to a natural quotient of a fusion product of Kirillov-Reshetikhin modules. One important part of the proof of the main result  shows that any truncated Weyl module is isomorphic to a Chari-Venkatesh module and explicitly describes the corresponding family of partitions. This leads to further results in the case that $\lie g=\lie{sl}_2$ related to Demazure flags and chains of inclusions of truncated Weyl modules.
\end{abstract}

\maketitle

\section{Introduction}

The concept of Weyl modules in the realm of finite dimensional representation theory of classical and quantum affine algebras was introduced by  Chari and Pressley in \cite{cp:weyl}. The definition, given via generators and relations, was inspired by the similar notion in modular representation theory of algebraic groups. In the years that followed, the notion was extended to other algebras sharing some similarity with the affine Kac-Moody algebras such as algebras of the form $\lie g\otimes A$ where $\lie g$ is a symmetrizable Kac-Moody algebra or a Lie super algebra and $A$ is a commutative associative algebra with unit (see \cite{cls,cfk:cat,fl:multw,frs,jm:weyl,nesa:sur} and references therein). After \cite{fl:multw}, there is a distinction between two kinds of Weyl modules: the local and the global ones (both appeared in \cite{cp:weyl}, but only the former under the terminology Weyl module). Since this work is concerned only with local Weyl modules, we shall simply say Weyl modules.  

The context of the current algebra $\lie g[t]=\lie g\otimes\mathbb C[t]$ with $\lie g$ a finite-dimensional simple Lie algebra is the most studied for several reasons. On one hand, some questions about the structure of the Weyl modules for quantum and classical affine Kac-Moody algebras can be reformulated as questions about certain quotients of the graded Weyl modules for $\lie g[t]$. In particular, this connects the theory of Weyl modules to those of Demazure modules (\cite{cl:wdf, foli:wdkr}) and fusion products (in the sense of \cite{fl:kosver}). On the other hand, the study of the category of graded finite-dimensional representations of the current algebra is motivated by applications in mathematical physics, algebraic geometry and geometric Lie theory,  as well as combinatorics.  

The definition of Weyl modules for generalized current algebras can be explained as follows.  Given a triangular decomposition of $\lie g$, say $\lie g=\lie n^-\oplus\lie h\oplus \lie n^+$ where $\lie h$ is a Cartan subalgebra, and a $\mathbb Z_{\ge 0}$-graded, commutative, associative algebra with unit $A$,	consider the induced decomposition: $$\lie g\otimes A=\lie n^- \otimes A\oplus \lie h \otimes A\oplus \lie n^+\otimes A.$$
Any linear functional $\lambda$ on $\lie h$ can be extended  to one on  $(\lie h \oplus \lie n^+)\otimes A$ by setting it to be zero on $\lie h\otimes A_+ \oplus \lie n^+\otimes A$, where $A_+$ is the ideal spanned by the positive degree elements. One can then consider  the Verma type module  $M(\lambda)$ associated to the above  induced triangular decomposition, which is naturally $\mathbb Z_{\ge 0}$-graded. If $\lambda$ is a dominant integral weight, the irreducible quotient of $M(\lambda)$ is finite-dimensional: it is the corresponding finite-dimensional simple $\lie g$-module $V(\lambda)$  on which $\lie g \otimes A_+$ acts trivially. It turns out that $M(\lambda)$ has other finite-dimensional quotients and the largest are called Weyl modules, i.e. any finite-dimensional quotient of $M(\lambda)$ is a quotient of some of the Weyl modules. 

 If the zero graded piece of $A$ is $\mathbb C$, e.g. when $A=\mathbb C[t]$ or $\mathbb C[t]/t^N\mathbb C[t]$ (where $N$ is a natural number), there is a unique graded Weyl module for each $\lambda$, up to isomorphism. Therefore, the graded Weyl module is the universal highest-weight module of highest weight $\lambda$ in the corresponding category of finite-dimensional graded modules. 
We let $W(\lambda)$ and $W_N(\lambda)$ denote the graded Weyl module for  $\lie g[t]  = \lie g \otimes  \mathbb C[t]$ and $\lie g[t]_N = \lie g \otimes \frac{\mathbb C[t]}{t^N\mathbb C[t]}$, respectively, and refer to $W_N(\lambda)$ as a truncated Weyl module.  One motivation for studying the truncated Weyl modules comes from a conjecture stated in \cite{cfs:schur} related to Schur positivity which, as seen in \cite{fou:nhi,kl:tor,rav:dcvf}, can be rephrased as a conjecture on the realization of truncated Weyl modules as fusion products of certain irreducible modules. In particular, a positive answer for this conjecture leads to a way of computing the characters of truncated Weyl modules, which is still an open problem in general.

To explain the conjecture, consider the set $P^+(\lambda , N)$ whose elements are $N$-tuples $$\bs \lambda = (\lambda _1, \ldots , \lambda _N)$$
of dominant integral weights adding up to $\lambda$.  A partial order on $P^+(\lambda , N)$ was defined in \cite{cfs:schur} and an algorithm for computing its maximal elements was described in \cite{fou:order}. It turns out that all maximal elements are in the same orbit under the obvious action of the symmetric group and, hence, the character of the fusion products 
$$V(\lambda_1)*\cdots* V(\lambda_N)$$
is independent of the choice of a maximal element $\bs{\lambda}$.  Recall also that, since $\lie g[t]_N$ is a graded quotient of $\lie g[t]$, every $\lie g[t]_N$-module can be regarded as a module for $\lie g[t]$. In particular,  we regard  $W_N(\lambda)$ as a $\lie g[t]$-module and, in that case, it is a quotient of $W(\lambda)$. Let $|\lambda|$ be the sum of the evaluations of $\lambda$ in the simple coroots of $\lie g$. 
 The following was stated in \cite{fou:nhi}:

\noindent\textbf{Conjecture:} 
Suppose $\bs{\lambda} = (\lambda_1,\dots,\lambda_N)$ is a maximal element of $P^+(\lambda,N)$. If $N\le|\lambda|$, there exists an isomorphism of $\mathbb Z$-graded $\lie g[t]$-modules: $W_N(\lambda)\cong V(\lambda_1)*\cdots* V(\lambda_N)$.

The conjecture has been proved for particular values of $\lambda$ or $N$. For instance, it follows from the results of \cite{fou:nhi} for $N=2$, $\lie g=\lie{sl}_n$, and $\lambda$ any multiple of a fundamental weight. Other cases for general $N$, with restrictions on $\lambda$ or $\lie g$ were proved in \cite{kl:tor,rav:dcvf}. Our main result (Theorem \ref{t:min}) extends the list of known cases for which the conjecture holds. Namely, we prove it for every $N>1$ in the case that $\lambda$ is a multiple of a ``small'' fundamental weight. More precisely, a fundamental weight $\omega_i$ is ``small' if the coordinate of the highest root of $\lie g$ in the direction of the simple root $\alpha_i$ is $1$. Since all minuscule fundamental weights are small in this sense when $\lie g$ is simply laced, this gives an alternative proof for the case $\lie g=\lie{sl}_n$ and $N=2$ previously established in \cite{fou:nhi}. In fact, the proof presented here is completely different from that  of \cite{fou:nhi} as it relies on the theory of Chari-Venkatesh (CV) and Kirillov-Reshetikhin (KR) modules. 

The CV modules, introduced in \cite{cv:cvm}, are graded quotients of Weyl modules associated to certain families of partitions indexed by the set $R^+$ of  positive roots of $\lie g$.  A fact used crucially in the proof of Theorem \ref{t:min} is that  every truncated Weyl module is isomorphic to a CV module (see Theorem \ref{t:cvtruncated} which had been proved for $\lie g=\lie{sl}_2$ in \cite{kl:tor}).  The associated family of partitions $\xi^\lambda_N$ is explicitly described as follows. Given $\alpha\in R^+$, let $h_\alpha$ be the associated coroot and let $q$ and $r$ be the quotient and remainder of the division of $\lambda(h_\alpha)$ by $N$. Then, the partition associated to $\alpha$ is
\begin{equation*}
  ((q+1)^{(r)}, q^{(N-r)}),
\end{equation*} 
where the exponent indicates the number of times that part is repeated. Together with results from \cite{bcsv:dfcpmt,cv:cvm, cssw:demflag}, Theorem \ref{t:cvtruncated} leads to further results in the case that $\lie g = \lie {sl}_2$ related to Demazure flags and chains of inclusions of truncated Weyl modules. For instance, from the description of $\xi^\lambda_N$ and results from \cite{cv:cvm}, one can immediately identify the truncated Weyl modules which are isomorphic to Demazure modules. Otherwise, the results of \cite{bcsv:dfcpmt, cssw:demflag} allows us to study Demazure flags for truncated Weyl modules since every CV module (for $\lie g=\lie{sl}_2$) admits a Demazure flag.

The KR modules were originally considered in the quantum setting motivated by mathematical physics \cite{kr} and remain objects of intense studies (see \cite{her:krc,naoi:tpkr} and references therein). They can also be seen as minimal affinizations, in the sense of \cite{cha:minr2}, having a multiple of a fundamental weight as  highest weight. The graded KR modules are the so called graded limits of the original KR modules, in the sense of \cite{mou:res}. The proof of Theorem \ref{t:min} also relies on a result from \cite{naoi:tpkr} which gives a presentation in terms of generators and relations for fusion products of graded KR modules. This result also allows us to take a further step towards proving the conjecture with $\lambda$ being a multiple of any fundamental weight. Namely, in Proposition \ref{p:W=T}, we prove that a truncated Weyl module whose highest weight is a multiple of the fundamental weight $\omega_i$ is isomorphic to a ``natural'' quotient of the fusion product of the KR modules associated to the partition indexed by the corresponding simple root in the sense of Theorem \ref{t:cvtruncated}.

The text is organized as follows. In Section \ref{s:main}, after a brief review on simple Lie algebras, current algebras and their representations, fusion products, CV modules, and KR modules, we state our mains results. The proofs are given in Section \ref{s:pf}, after reviewing further results and properties of fusion products and CV modules which are needed in the arguments. The final section contains the aforementioned results related to Demazure flags and inclusions of truncated Weyl modules for $\lie g = \lie {sl}_2$.

\section{The Main Results}\label{s:main}

We shall use the symbol $\diamond$ to mark the end of remarks, examples, and statements of results whose proofs are postponed. The symbol \qedsymbol\ will mark the end of proofs as well as of statements whose proofs are omitted. The sets of integers and complex numbers will be denoted by $\mathbb Z$ and $\mathbb C$, respectively, and the notation $\mathbb Z_{\ge m}, m\in \mathbb Z$, is defined in the obvious way.

\subsection{Current Algebras and Their Irreducible Modules}

Let $\lie g$ be a finite-dimensional simple Lie algebra over $\mathbb C$, fix a Cartan subalgebra $\lie h \subset \lie g$ as well as a Borel subalgebra $\lie b \supseteq \lie h$. Let $R, R^+ \subseteq \lie h^\ast$ be the sets of roots and positive roots, respectively, corresponding to these choices and denote by $\Delta = \{ \alpha _1, \ldots , \alpha _n\}$ the corresponding set of simple roots. Let also $\omega _1, \ldots , \omega _n$ denote the corresponding fundamental weights. For convenience, set $I = \{1, \ldots , n\}$. The root and weight lattices and their positive cones will be denoted by $Q, Q^+, P, P^+$, respectively. The highest root will be denoted by $\theta$ and the highest short root by $\vartheta$. We use the convention that, when $\lie g$ is simply laced, all roots are short and long simultaneously. Given $\eta=\sum_i a_i\alpha_i\in Q$ and $i\in I$, set
\begin{equation*}
\het_i(\eta) = a_i \qquad\text{and}\qquad \het(\eta) = \sum_{i\in I}\het_i(\eta).
\end{equation*}

Fix a Chevalley basis $\{x^\pm_\alpha, h_i:\alpha\in R^+,i\in I \}$ and set $h_\alpha = [x_\alpha^+,x_\alpha^-], \alpha\in R^+,$ and $x_i^\pm = x_{\alpha_i}^\pm, i\in I$. In particular, $h_i=h_{\alpha_i}$. Let $c_{i,j} = \alpha_j(h_i)$ be the Cartan matrix of $\lie g$ and $d_1,\dots,d_n$ be positive relatively prime integers such that $DC$ is symmetric, where $D={\rm diag}(d_1,\dots,d_n)$. Set also
\begin{equation*}
P_{sym}^+ = \{\lambda\in P^+: d_i|\lambda(h_i)\}=\bigoplus_{i\in I}\mathbb Z_{\ge 0}\ d_i\omega_i. 
\end{equation*}
Recall that, for all $\alpha\in R^+$, $\lie g_{\pm \alpha}=\mathbb Cx_\alpha^\pm$ is the associated root space and, setting $\lie n^\pm = \sum_{\alpha\in R^+} \lie g_{\pm\alpha}$, we have 
\begin{equation*}
\lie g = \lie n^-\oplus\lie h\oplus\lie n^+ \qquad\text{and}\qquad \lie b = \lie h\oplus\lie n^+.
\end{equation*} 
Also, the vector subspace $\lie{sl}_\alpha$ spanned by $x_\alpha^\pm, h_\alpha$ is a subalgebra isomorphic to $\lie{sl}_2$.

For any Lie algebra $\lie l$ and commutative associative algebra $A$, the vector space $\lie l\otimes A$  can be equipped with a Lie algebra structure by setting
$$[x\otimes a , y\otimes b]= [x,y] \otimes ab \qquad \mbox{for all}\qquad x, y \in \lie l, \quad a,b \in A.$$
If $A$ has an identity element, then the subspace $\lie l\otimes 1$ is a Lie subalgebra of $\lie l\otimes A$ isomorphic to $\lie l$ and we identify $\lie l$ with this subalgebra. If $A$ is graded, then $\lie g\otimes A$ inherits the gradation in the obvious way.  
In the case that $A=\mathbb C[t]$ is the polynomial ring in one variable, this algebra is called the current algebra over $\lie l$ and will be denote by $\lie l[t]$. If $A=\mathbb C[t]/(t^N)$ for some $N\in\mathbb Z_{\ge 0}$, the algebra $\lie l\otimes A$ is called the truncated current Lie algebra of nilpotence index $N$ and will be denoted by $\lie l[t]_N$. For convenience, we set $\lie l[t]_\infty=\lie l[t]$. Note that
\begin{equation}\label{e:triangdec}
\lie g\otimes A= \lie n^-\otimes A\ \oplus\ \lie h\otimes A\ \oplus\ \lie n^+\otimes A.
\end{equation}

Given $a\in \mathbb C$, let $\ev_a: \lie g[t] \rightarrow \lie g$ be the evaluation map $x\otimes f(t) \mapsto f(a)x$, which is a Lie algebra homomorphism. Thus, if $V$ is a $\lie g$-module, we can consider the $\lie g[t]$-module $\ev_a V$ obtained by pulling-back the action of $\lie g$ to one of $\lie g[t]$ via $\ev_a$. Modules of this form are called evaluation modules. Note that $\ev_a V$ is simple if and only if $V$ is simple. 
By abuse of notation, we shall identify $V$ with $\ev_0 V$. Given a $\lie g$-module $V$ and $\mu\in P$, we denote the associated weight space by $V_\mu$. Set also
\begin{equation*}
|\mu| = \sum_{i\in I} \mu(h_i) \qquad\text{for}\qquad \mu\in P.
\end{equation*}

Given $\lambda\in P^+$, we will denote by $V(\lambda)$ an irreducible $\lie g$-module of highest weight $\lambda$ and by $V_a(\lambda)$ the corresponding evaluation module. Recall that $V(\lambda)$ is generated by a vector $v$ satisfying the following defining relations:
\begin{equation*}
\lie n^+v =0, \quad hv=\lambda(h)v, \quad (x_i^-)^{\lambda(h_i)+1}v = 0 \quad\text{for all}\quad h\in\lie h, i\in I.
\end{equation*}
In particular, when $v$ is regarded as an element of $V_a(\lambda)$, we have
\begin{equation}\label{e:defrelweylev}
\lie n^+[t]v=0 \qquad\text{and}\qquad (h\otimes t^r)v = a^r\lambda(h)v \quad\text{for all}\quad h\in\lie h, r\in\mathbb Z_{\ge 0},
\end{equation}
which shows that $V_a(\lambda)$ is a highest-weight module with respect to the decomposition \eqref{e:triangdec}.
Notice also that $V_a(\lambda)$ is a $\mathbb Z$-graded $\lie g[t]$-module if and only if $a=0$. Moreover, if $N>1$, $V_a(\lambda)$ factors to a $\lie g[t]_N$-module if and only if $a=0$.
Given $k\ge 0, \lambda_1,\dots,\lambda_k\in P^+\setminus\{0\}$, and $a_1,\dots,a_k\in\mathbb C$, it is well-known that 
\begin{equation}\label{e:irrastp}
V_{a_1}(\lambda_1)\otimes\cdots\otimes V_{a_k}(\lambda_k) \quad\text{is irreducible}\quad \Leftrightarrow\quad a_i\ne a_j \quad\text{for}\quad i\ne j.
\end{equation}
Moreover, every irreducible finite-dimensional $\lie g[t]$-module is isomorphic to a unique tensor product of this form. Note that, if $v_j\in V(\lambda_j)_{\lambda_j}, 1\le j\le k$, and $v=v_1\otimes\cdots\otimes v_k$, it follows from \eqref{e:defrelweylev} that
\begin{equation}\label{e:defrelweylgen}
 (h\otimes t^r)v = \left(\sum_{j=1}^k a_j^r\lambda_j(h)\right)v \qquad\text{for all}\qquad h\in\lie h,\ r\in\mathbb Z_{\ge 0}.
\end{equation}

Denote by $\mathcal G$ and $\mathcal G_N$ the categories of graded finite-dimensional $\lie g[t]$-modules and $\lie g[t]_N$-modules, respectively, where the morphisms are those preserving grades. By pulling-back by the natural epimorphism of Lie algebras $\lie g[t]\to\lie g[t]_N$, every object from $\mathcal G_N$ can be regarded as an object in $\mathcal G$. Evidently, every object from $\mathcal G$ arises in this way for some $N\in\mathbb Z_{> 0}$.
We shall denote the $k$-th graded piece of a $\mathbb Z$-graded vector space $V$ by $V[k]$. It will be convenient to introduce the notation
\begin{equation*}
V_+ = \bigoplus_{k>0} V[k].
\end{equation*}
For $m\in\mathbb Z$, we consider the functor $\tau_m$ on the category of $\mathbb Z$-graded vector spaces which shifts the grades by the rule
\begin{equation*}
\tau_m(V)[k] = V[k-m].
\end{equation*}
This functor induces endofunctors of  $\mathcal G_N, N\in\mathbb Z_{>0}\cup\{\infty\}$,  by shifting the grades and keeping the action of $\lie g[t]_N$ unaltered.  For $\lambda\in P^+$ and $m\in\mathbb Z$, set
\begin{equation*}
V(\lambda,m) = \tau_m(V_0(\lambda)).
\end{equation*}
It follows from the comment after \eqref{e:irrastp} that a simple object in $\mathcal G_N$ is isomorphic to a unique module of the form $V(\lambda,m)$. For more on the simple finite-dimensional representations of algebras of the form $\lie g\otimes A$, see the survey \cite{nesa:sur} and references therein. 

Given $a\in\mathbb C$, consider the Lie algebra automorphism $\zeta_a$ of $\lie g[t]$ induced by $t\mapsto t+a$. Then, given a $\lie g[t]$-module $V$, denote by $V_a$ the pullback of $V$ by $\zeta_a$. Note that, if $V\in\mathcal G$ and $a\ne 0$, then $V_a$ is not a graded module. One easily checks that, if $V=\ev_0W$ for some $\lie g$-module $W$, there exists an isomorphism of $\lie g[t]$-modules $V_a \cong \ev_a W$.

\subsection{Truncated Local Weyl Modules}

The local Weyl modules are certain modules for algebras of the form $\lie g\otimes A$ which are highest-weight with respect to the decomposition \eqref{e:triangdec}.  Given a unital associative algebra $A$ and $\bs\omega\in(\lie h\otimes A)^*$, one can consider the Verma module $M(\bs\omega)$ which is generated by a vector $v$ satisfying the highest-weight relations
\begin{equation*}
(\lie n^+\otimes A) v = 0 \qquad\text{and}\qquad xv = \bs{\omega}(x)v \quad\text{for all}\quad x\in\lie h\otimes A
\end{equation*}
as defining relations. Under certain conditions on $\bs{\omega}$ (see \cite{cfk:cat} for details), $M(\bs{\omega})$ admits nonzero finite-dimensional quotients. In the case $A=\mathbb C[t]$, it follows from the classification of the finite-dimensional irreducible modules and \eqref{e:defrelweylgen}  that these conditions are equivalent to the existence of  $k\ge 0, \lambda_1,\dots,\lambda_k\in P^+\setminus\{0\}$, and distinct $a_1,\dots,a_k\in\mathbb C$ such that 
\begin{equation*}
 \bs{\omega}(h\otimes t^r) = \sum_{j=1}^k a_j^r\lambda_j(h) \qquad\text{for all}\qquad h\in\lie h,\ r\in\mathbb Z_{\ge 0}.
\end{equation*}
In that case,  $\bs{\omega}|_{\lie h}\in P^+$ and the local Weyl module $W(\bs{\omega})$ can be defined as the quotient of  $M(\bs{\omega})$ by the submodule generated by
\begin{equation*}
(x_i^-)^{\bs{\omega}(h_i)+1} v \qquad\text{for all}\qquad i\in I.
\end{equation*}
It turns out that $W(\bs{\omega})$ is finite-dimensional and every other finite-dimensional quotient of $M(\bs{\omega})$ is also a quotient  of  $W(\bs{\omega})$. Note that $M(\bs{\omega})$ is graded if and only if $\bs\omega(\lie h[t]_+) = 0$ and it admits finite-dimensional quotients if and only if $\bs{\omega}|_\lie h\in P^+$. Moreover, in that case, the corresponding local Weyl module is also graded and we denote it simply by $W(\lambda)$ where $\lambda = \bs{\omega}|_\lie h$. In other words, the graded local Weyl module associated to $\lambda\in P^+$ is the $\lie g[t]$-module generated by a vector $v$ satisfying the following defining relations
\begin{equation}\label{e:defgrweyl}
\begin{aligned}
\lie n^+[t]v=0, & \qquad (h\otimes t^r)v = \delta_{r,0}\lambda(h)v \quad\text{for all}\quad h\in\lie h, r\in\mathbb Z_{\ge 0},\\ & \text{and}\qquad (x_i^-)^{\lambda(h_i)+1} v=0 \qquad\text{for all}\qquad i\in I.
\end{aligned}
\end{equation}
Given $N\in\mathbb Z_{>0}$, we define the truncated local Weyl module $W_N(\lambda)$ as the $\lie g[t]_N$-module generated by a vector $v$ satisfying \eqref{e:defgrweyl} as defining relations. It follows from \cite{cfk:cat} that $W_N(\lambda)$ is finite-dimensional and every finite-dimensional graded highest-weight (with respect to \eqref{e:triangdec}) $\lie g[t]_N$-module is a quotient of $W_N(\lambda)$ for a unique $\lambda\in P^+$. For convenience, we set $W_\infty(\lambda)=W(\lambda)$.

Since every $\lie g[t]_N$-module can be naturally regarded as a $\lie g[t]$-module, the universal property of $W(\lambda)$ immediately implies that we have an epimorphism of $\lie g[t]$-modules:
\begin{equation*}
W(\lambda)\twoheadrightarrow W_N(\lambda).
\end{equation*}
If $v\in W(\lambda)_\lambda\setminus\{0\}$, one can easily check that the kernel of this map is the submodule of $W(\lambda)$ generated by
\begin{equation}\label{e:relfor WN}
(x\otimes t^N)v \qquad\text{for}\qquad x\in\lie n^-.
\end{equation}
The details can be found in \cite{mar:PhD}.
Since $\lie h[t]_+v = 0$, it follows that $(x\otimes t^k)v$ is in the kernel for all $k\ge N$. In particular, we have epimorphisms of $\lie g[t]$-modules
\begin{equation}\label{e:truncproj}
W_{N}(\lambda)\twoheadrightarrow W_{N'}(\lambda) \quad\text{if}\quad N\ge N'.
\end{equation}
It is well known that, for all $\alpha\in R^+$,
\begin{equation}\label{e:endofweyl}
(x_\alpha^-\otimes t^r)v = 0 \quad\text{if}\quad r\ge \lambda(h_\alpha).
\end{equation}
Hence, 
\begin{equation*}%\label{e:weylt=nt}
W_N(\lambda)\cong W(\lambda) \quad\text{if}\quad N\ge\max\{\lambda(h_\alpha):\alpha\in R^+\}
\end{equation*}
or, equivalently,
\begin{equation}\label{e:Nlargedef}
W_N(\lambda)\cong W(\lambda) \quad\text{if}\quad N\ge\lambda(h_\vartheta).
\end{equation}

\subsection{Fusion Products}  
We now review the notion of fusion products defined in \cite{fl:kosver} (for more details see \cite{cl:wdf,rav:dcvf}).
If $V$ is a cyclic $\lie g[t]$-module and $v$ generates $V$, define a filtration on $V$ by
\begin{equation*}
 F^rV = \sum _{0 \leq s \leq r} U(\lie g[t])[s] v
\end{equation*}
where $U(\lie g[t])[s]$ denotes the $s$-th graded piece of $U(\lie g[t])$ for the $\mathbb Z$-grading induced from that of $\mathbb C[t]$.
For convenience of notation, we set $F^{-1}V$ to be the zero space. Then, $\displaystyle \gr\ V := \bigoplus _{r \geq 0} \dfrac{F^rV}{F^{r-1}V}$ becomes a cyclic graded $\lie g[t]$-module with action given by 
\begin{equation}
(x \otimes t^s)(w+ F^{r-1}V)= (x \otimes t^s)w+ F^{r+s-1}V 
\end{equation}
for all $x \in \lie g,$ $w \in F^rV$, $r, s \in \mathbb Z _{\geq 0}$.

Given $k\in\mathbb Z_{>0}$, a family of distinct complex numbers $a_1,\dots,a_k$,  and generators $v_1, \dots, v_k$ of cyclic objects $V^1,\dots, V^k$ from $\mathcal G$, it was proved in \cite{fl:kosver} that
\begin{equation*}%\label{e:tpofcyclic}
V_{a_1}^1\otimes \cdots\otimes V_{a_k}^k
\end{equation*}
is generated by $v = v_1\otimes\cdots\otimes v_k$. The fusion product of $V^1,\dots, V^k$  with respect to the parameters  $a_1,\dots,a_k$ (it also depends on the choices of cyclic generators), which will be denoted by
\begin{equation*}
V^1 _{a_1} * \cdots * V^k _{a_k},
\end{equation*}
is the associated graded $\lie g[t]$-module corresponding to the filtration on $V_{a_1}^1\otimes \cdots\otimes V_{a_k}^k$ as defined above. Note that we have an isomorphism of $\lie g$-modules
\begin{equation*}
V^1\otimes\cdots\otimes V^k \cong_{\lie g} V^1 _{a_1} * \cdots * V^k _{a_k}.
\end{equation*}
It was conjectured in \cite{fl:kosver} that, under certain conditions, the fusion product does not actually depend on the choice of the parameters  $a_1,\dots,a_k$. Motivated by this conjecture, it is usual to simplify notation and write $V^1 * \cdots * V^k$ instead of $V^1 _{a_1} * \cdots * V^k _{a_k}$. This conjecture has been proved in some special cases  (see \cite{cl:wdf, ff:qchtp, fl:kosver, foli:wdkr, kl:tor,naoi:tpkr} and references therein). In all these special cases, each $V^j$ is a quotient of a graded local Weyl module and the cyclic generator $v_j$ is a highest-weight generator. All cases relevant to us are of this form and, hence, we make no further mention about the choice of cyclic generators. 
In particular, it is was proved in \cite{cl:wdf} for $ \lie g = \lie{sl}_n$, in \cite{foli:wdkr} for simply laced $\lie g$, and in \cite{naoi:tpkr} in general, that we have an isomorphism of graded $\lie g[t]$-modules:
\begin{equation}\label{e:weylfusion}
W(\lambda_1)*\cdots*W(\lambda_k) \cong W(\lambda) \quad\text{if}\quad \lambda = \lambda_1+\cdots+\lambda_k.
\end{equation}
Note that this is equivalent to saying that, for all $\lambda\in P^+$, we have
\begin{equation}\label{e:weylfusfund}
W(\lambda)\cong W(\omega_1)^{*\lambda(h_1)}*\cdots*W(\omega_n)^{*\lambda(h_n)}.
\end{equation}
We now recall a conjectural generalization of \eqref{e:weylfusion} for truncated Weyl modules stated in \cite{cfs:schur, fou:nhi, kl:tor} which is the subject of the first of our main results.

Given $\lambda\in P^+$ and $N\in\mathbb Z_{>0}$, let $P^+(\lambda,N)$ be the subset of $(P^+)^N$ consisting of elements $\bs{\lambda} = (\lambda_1,\dots,\lambda_N)$ such that
\begin{equation*}
\lambda_1+\cdots+\lambda_N = \lambda.
\end{equation*} 
Given $\alpha\in R^+$ and $1\le k\le N$, define
\begin{equation}
r_{\alpha,k}(\bs{\lambda}) = \min\{(\lambda_{i_1}+\cdots+\lambda_{i_k})(h_\alpha): 1\le i_1<\cdots<i_k\le N\}.
\end{equation}
Then, equip $P^+(\lambda,N)$ with the partial order defined by
\begin{equation*}
\bs{\mu}\le\bs{\lambda} \quad\Leftrightarrow\quad r_{\alpha,k}(\bs{\mu})\le r_{\alpha,k}(\bs{\lambda}) \quad\text{for all}\quad \alpha\in R^+,\ 1\le k\le N.
\end{equation*}
It turns out that the maximal elements of $P^+(\lambda,N)$ form a unique orbit under the obvious action of the symmetric group $S_N$ on  $P^+(\lambda,N)$. Note that, if $N\ge|\lambda|$, an element of $P^+(\lambda,N)$ is maximal if and only if all its nonzero entries are fundamental weights. Motivated by this uniqueness and the results on Schur positivity from \cite{cfs:schur}, the following conjecture was formulated by the first author, Chari and Sagaki, stated first in \cite{fou:nhi}:

\begin{con}\label{cj:KusLit}
	Let $N\in\mathbb Z_{>0}, \lambda\in P^+$, and suppose $\bs{\lambda} = (\lambda_1,\dots,\lambda_N)$ is a maximal element of $P^+(\lambda,N)$. If $N\le|\lambda|$, $W_N(\lambda)\cong V(\lambda_1)*\cdots* V(\lambda_N)$ as graded $\lie g[t]$-modules.\endd
\end{con}

Recall that $\lambda\in P^+$ is said to be minuscule if $\{\mu\in P^+:\mu<\lambda\}=\emptyset$. All nonzero minuscule weights are fundamental weights and, if $\lie g$ is of type $A$, then all fundamental weights are minuscule. Conjecture \ref{cj:KusLit} has been proved in the following special cases: 
\begin{enumerate}[(1)]
	\item for simply laced $\lie g$, $\lambda=m\theta$ for some $m\ge 0$, and $N=|\lambda|$ in \cite{rav:dcvf};
	\item for $\lie g$ of type $A$, $N=2$, and $\lambda=m\omega_i$ for some $i\in I$ in \cite{fou:nhi};
	\item for $\lambda=N\mu+\nu$ with $\mu\in P^+_{sym}$ and $\nu$ minuscule in \cite{kl:tor}.	
\end{enumerate}

Our main result generalizes and provides an alternate proof for item (2) above:

\begin{thm}\label{t:min}
	Suppose $i\in I$ satisfies $\het_i(\theta)=1$. Then, Conjecture \ref{cj:KusLit} holds for $\lambda=m\omega_i$ for all $m\ge 0$. 
	\endd
\end{thm}

The proof of Theorem \ref{t:min} will rely on results about Kirillov-Reshetikhin and Chari-Venkatesh  modules which we review in the next subsections. 

Note that Conjecture \ref{cj:KusLit} is not a complete generalization of \eqref{e:weylfusfund} since we may have $|\lambda|<\lambda(h_\vartheta)$ (cf. \eqref{e:Nlargedef}). Regarding the region $|\lambda|\le N\le \lambda(h_\vartheta)$, it was proved in  \cite{rav:dcvf} for simply laced $\lie g$ that 
\begin{equation}
W_N(m\theta)\cong W(\theta)^{*(N-m)}* V(\theta)^{*(2m-N)}.
\end{equation}
We shall address this region in the future.

\subsection{Truncated Weyl Modules as Chari-Venkatesh Modules}

Given a sequence $\bsr m = (m_j)_{j\in\mathbb Z_{>0}}$ of nonnegative integers, we let $\supp(\bsr m)=\{j: m_j\ne 0\}$. We denote by $\mathscr P$ the set of non-increasing monotonic sequences with finite-support and refer to the elements of $\mathscr P$ as partitions. For any sequence $\bsr m$ with finite support, we denote by $\bsru m$ the partition obtained from $\bsr m$ by re-ordering its elements. Given $\bsr m\in\mathscr P$, set
\begin{equation*}
\ell(\bsr m) = \max\{j:m_j\ne 0\} \qquad\text{and}\qquad |\bsr m| = \sum_{j\ge 1} m_j.
\end{equation*}
If $|\bsr m|=m$, then $\bsru m$ is said to be a partition of $m$. We denote by $\mathscr P_m$ the set of partition of $m$. The element $m_j$ of a partition $\bsr m$ will be often referred to as the $j$-th part of $\bsr m$. Hence, $\ell(\bsr m)$, which is often referred to as the length of $\bsr m$, is the number of nonzero parts of $\bsr m$. Given distinct nonnegative integers $k_1>k_2\cdots>k_l$ and $a_1,\dots,a_l$, we denote by
\begin{equation*}
 (k_1^{(a_1)},\dots, k_l^{(a_l)})
\end{equation*}
the partition where each $k_j$ is repeated $a_j$ times.

Given $\lambda\in P^+$, a family of partitions $\xi =(\xi(\alpha ))_{\alpha\in R^+}$ indexed by $R^+$ is said to be $\lambda$-compatible if 
\begin{equation*}
\xi(\alpha )\in\mathscr P_{\lambda(h_\alpha)} \qquad\text{for all}\qquad \alpha\in R^+.   
\end{equation*}
We will denote by $\mathscr P_\lambda$ the set of $\lambda$-compatible family of partitions. Given $\xi\in\mathscr P_\lambda$,  the CV module $CV(\xi)$  is the quotient of $W(\lambda)$ by the submodule generated by
\begin{equation}\label{e:cvdef}
(x_\alpha^+ \otimes t)^s (x_\alpha^-)^{s+r} w, \quad \alpha \in R^+,\ s,r\ne 0,\ s+r > rk + \sum _{j > k} \xi(\alpha )_j
\end{equation}
for some $k>0$. These modules were introduced in \cite{cv:cvm}.

Fix $\lambda\in P^+$ and $N\in\mathbb Z_{>0}\cup\{\infty\}$ and, for each $\alpha\in R^+$, let $q_\alpha$ and $p_\alpha$ be the unique integers such that
\begin{equation}\label{e:eucdiv}
\lambda(h_\alpha) = Nq_\alpha + p_\alpha \quad\text{and}\quad 0\le p_\alpha<N,
\end{equation}
where we understand $q_\alpha=0$ and $p_\alpha=\lambda(h_\alpha)$ if $N=\infty$. Then, consider the element $\xi_N^\lambda\in\mathscr P_\lambda$ given by
\begin{equation}\label{e:trxi}
\xi_N^\lambda(\alpha) = ((q_\alpha+1)^{(p_\alpha)},q_\alpha^{(N-p_\alpha)}).
\end{equation}
The following theorem will be crucial in the proof of Theorem \ref{t:min}.

\begin{thm}\label{t:cvtruncated}
The modules $CV(\xi_N^\lambda)$ and $W_N(\lambda)$ are isomorphic graded $\lie g[t]$-modules.	\endd
\end{thm}

This theorem was proved in \cite[Theorem 4.3]{kl:tor} for $\lie g=\lie {sl}_2$ and the general case follows easily from that case as we shall see in  Section  \ref{ss:cvtruncated} (see also \cite[Theorem 5]{she}). 
Note that, if $N\ge\lambda(h_\vartheta)$, then $\xi_N^\lambda(\alpha)=(1^{(\lambda(h_\alpha))})$ for all $\alpha\in R^+$ and Theorem \ref{t:cvtruncated} follows from the results of \cite{cv:cvm} (see also \cite[Proposition 2.2.5]{mar:PhD}).   In Section \ref{s:sl2}, we will obtain further results about truncated Weyl modules in the case $\lie g=\lie{sl}_2$ as consequences of Theorem \ref{t:cvtruncated} together with results from \cite{bcsv:dfcpmt,cssw:demflag}.

\subsection{Kirillov-Reshetikhin Modules}
The proof of Theorem \ref{t:min} will rely on Theorem \ref{t:cvtruncated} as well as on a result about fusion products of graded Kirillov-Reshetikhin (KR) modules.  Following \cite{cha:fer}, the graded KR module $KR(m\omega_i), m\in\mathbb Z_{\ge 0}, i\in I$, is defined as the quotient of $W(m\omega_i)$ by the submodule generated by
\begin{equation*}
 (x_i^-\otimes t)v \qquad\text{with}\qquad v\in W(m\omega_i)_{m\omega_i}\setminus\{0\}. 
\end{equation*}
These modules are the graded limits of the KR modules for quantum affine algebras (see \cite{cm:kr,mou:res,naoi:tpkr} and references therein). 
It follows from \cite{cha:fer} (see also \cite{cm:kr,foli:wdkr}) that
\begin{equation}\label{e:skr}
\het_i(\theta) = 1 \qquad\Rightarrow\qquad KR(m\omega_i)\cong V(m\omega_i) \quad\text{for all}\quad m\ge 0.
\end{equation}

Given $N>0, \bsr i = (i_1,\dots, i_N)\in I^N$ and  $\bsr m = (m_1,\dots, m_N)\in\mathbb Z_{> 0}^N$, set 
\begin{gather}\notag
KR_{\bsr i}(\bsr m)  = KR(m_1\omega_{i_1})*\cdots* KR(m_N\omega_{i_N})\\ \label{e:fusKR} \quad\text{and}\quad\\ \notag 
V_{\bsr i}(\bsr m) = V(m_1\omega_{i_1})*\cdots* V(m_N\omega_{i_N}),
\end{gather}
for some choice of parameters in the definition of fusion products. In the case that $i_j=i$ for  all  $1\le j\le N$ and some $i\in I$, we write $KR_{i}(\bsr m)$ in place of $KR_{\bsr i}(\bsr m)$ and we similarly define $V_i(\bsr m)$. A presentation for $KR_{\bsr i}(\bsr m)$ in terms of generators and relations was obtained in \cite{naoi:tpkr}. In particular, the independence of $KR_{\bsr i}(\bsr m)$ on the choice of parameters for the fusion product follows. Together with Theorem \ref{t:cvtruncated}, this presentation will imply the following Corollary which, together with \eqref{e:skr}, leads to a proof of Theorem \ref{t:min}.

\begin{cor}\label{c:KR>>W}
	Let $m,N>0, i\in I$, and $\bsr m = \xi^{m\omega_i}_N(\alpha_i)$. Then, there exists an epimorphism of graded $\lie g[t]$-modules $KR_i(\bsr m)\to W_N(m\omega_i)$. \endd
\end{cor}

We also prove a further step towards a proof of Theorem \ref{t:min} without any hypothesis on $\lie g$ and $i$. To explain this result, introduce the following notation. For $i\in I, k\ge 0$, set
\begin{equation*}
R^+_{i,k} = \{\alpha\in R^+:\het_i(\alpha)=k\}.
\end{equation*}
Evidently, $R^+_{i,k}=\emptyset$ for $k\ge 2$ if $\het_i(\theta)=1$. Otherwise, by inspecting the root systems, one checks that  $R^+_{i,k}$ has a unique minimal element for $k\ge 2$ (with respect to the standard partial order on $P$). Let $v$ be a highest-weight generator for $KR_i(\bsr m)$ and denote by $T_i(\bsr m)$ the quotient by the submodule generated by 
\begin{equation}\label{e:Tidef}
(x_{\alpha}^-\otimes t^N)v \qquad\text{with}\qquad \alpha=\min R^+_{i,k},\ k\ge 2.
\end{equation}
We will see that, if $\bsr m = \xi^{m\omega_i}_N(\alpha_i)$ and $N\le m$, there are epimorphisms of graded $\lie g[t]$-modules
\begin{equation}\label{e:theprojs}
T_i(\bsr m)\twoheadrightarrow W_N(m\omega_i)\twoheadrightarrow V_i(\bsr m).
\end{equation}
The first of these epimorphisms is a consequence of Corollary \ref{c:KR>>W} together with \eqref{e:relfor WN} and is valid for all $m,N$ while the second is a corollary of Proposition \ref{fusiontrunc} below.
The third of our main results is:
\begin{prop}\label{p:W=T}
If $\bsr m = \xi^{m\omega_i}_N(\alpha_i)$ for some $i\in I,m\ge 0,N>0$,  $W_N(m\omega_i)\cong T_i(\bsr m)$.\endd
\end{prop}

Proposition \ref{p:W=T} says that the first arrow in \eqref{e:theprojs} is an isomorphism, while Conjecture \ref{cj:KusLit} expects that the second is also an isomorphism for all $i\in I$ provided $N\le m$. This motivates:

\begin{con}\label{cj:new}
	Let $i\in I, m\ge 0$. Then, for every $\bsr m\in\mathscr P_m$, $V_i(\bsr m)$ is isomorphic to $T_i(\bsr m)$.\endd
\end{con}

Note that item (1) after Conjecture \ref{cj:KusLit}, together with Proposition \ref{p:W=T}, proves this Conjecture for types $DE$, $\bsr m =\xi^{N\omega_i}_N(\alpha_i)$, and $i=i_\theta$ where $i_\theta\in I$ satisfies $\theta=m\omega_{i_\theta}$ for some $m>0$. Evidently, Conjecture \ref{cj:new}, together with \eqref{e:theprojs}, implies Conjecture \ref{cj:KusLit}.

\section{Proofs}\label{s:pf}

\subsection{More About CV Modules}\label{ss:CV} In this subsection we review some facts and technical tools concerning CV modules.

Given $f(u)\in U(\lie g[t])[[u]]$ and $s\in\mathbb Z_{\ge 0}$, let $f(u)_s$ be the coefficient of $u^s$ in $f$. Let also 
\begin{equation*}
f(u)^{(r)}= \frac{1}{r!} f(u)^{r} \quad\text{for all}\quad r\ge 0. 
\end{equation*}
For $\alpha\in R^+$, set 
\begin{equation}
\bsr x^\pm_\alpha(u) = \sum_{k\ge 0} (x^\pm_\alpha\otimes t^k)u^{k+1}
\end{equation}
and
\begin{equation*}
x_\alpha^\pm(r,s) = (\bsr x_\alpha(u)^{(r)})_{r+s} \quad\text{for all}\quad r,s>0.
\end{equation*}
In other words,
\begin{equation}
x_\alpha^\pm(r,s) = \sum _ {(b_p) \in S(r,s)} (x_\alpha^\pm \otimes 1)^{(b_0)} (x^\pm_\alpha\otimes t ) ^{(b_1)} \ldots (x^\pm_\alpha\otimes t^s)^{(b_s)},
\end{equation}
where
$$S(r, s) = \{(b_p)_{0 \leq p \leq s} :\, \, b_p \in \mathbb Z_{\geq 0}, \, \, \sum _{p = 0} ^s b_p = r, \, \, \sum _{p= 0 } ^s p b_p = s \}. $$
Given $k\geq 0$, define also
\begin{gather}\notag
_kS(r,s) = \{ (b_p)_{0\leq p\leq s} \in S(r,s): \, \, \, b_p = 0 \, \, \, \mbox{if} \, \, \, p < k \} \\ 
\label{e:defcv2ndrel}\quad\text{and}\quad\\ \notag
_kx_\alpha^\pm(r,s) = \sum _{(b_p) \in _kS(r,s)} (x_\alpha^\pm\otimes t^k)^{(b_k)}\ldots (x^\pm_\alpha \otimes t^{s})^{(b_{s})}.
\end{gather}
Note that $$_kx^\pm_\alpha(r,kr)=(x^\pm_\alpha\otimes t^k)^{(r)}.$$
If $\alpha=\alpha_i$ for some $i\in I$, we may simplify notation and write $\bsr x_i^\pm (u)$, etc. The following is \cite[Lemma 6]{rav:dcvf}.

\begin{lem}\label{l:rav}
	Let $\lambda\in P^+, w\in W(\lambda)_\lambda\setminus\{0\}$, and $\xi\in\mathscr P_\lambda$. Then, $x_\alpha^-(r,s)w=0$ for all $\alpha\in R^+, r\ge \xi(\alpha)_1, s>0$ such that $s+r>rk + \sum_{j>k} \xi(\alpha)_j$ for some $k>0$. \hfill\qedsymbol
\end{lem}

Let $w$ and $\xi$ be as in the above lemma and denote by $CV'(\xi)$ the quotient of $W(\lambda)$ by the submodule generated by 
\begin{equation}\label{e:CV2ndrel}
x_\alpha^-(r,s)w \quad\text{for all}\quad \alpha\in R^+,\ r,s>0 \quad\text{s.t.}\quad s+r > rk + \sum_{j>k} \xi(\alpha)_j
\end{equation}
for some $k>0$. Consider also the quotient $CV''(\xi)$ of $W(\lambda)$ by the submodule generated by 
\begin{equation}\label{e:CV3rdrel}
	_kx^-_\alpha(r,s) w \quad\text{for all}\quad \alpha \in R^+,\ s, r, k>0 \quad\text{s.t.}\quad s+r > rk + \sum _{j>k} \xi(\alpha )_j.
\end{equation}
The following was proved in \cite{cv:cvm}.

\begin{prop}\label{p:cv2ndrel}
	The modules $CV'(\xi)$ and $CV''(\xi)$ are isomorphic to $CV(\xi)$.\hfill\qedsymbol
\end{prop}

We will often denote by $v_\xi$ a nonzero element of $CV(\xi)_\lambda$ for $\xi\in\mathscr P_\lambda$. It follows from the previous proposition that
\begin{equation}%\label{e:cv2ndrel}
(x^-_\alpha \otimes t^k)^{(r)}v_\xi = 0 \quad\text{for all}\quad\alpha \in R^+, k,r> \text{ s.t. } r> \sum _{j>k} \xi(\alpha)_j.
\end{equation}
In particular, since $\sum _{j > k} \xi(\alpha)_j = 0$ for all $k\ge \ell(\xi(\alpha))$,
\begin{equation}\label{e:cv2ndrelparts}
(x^-_\alpha \otimes t^k)v = 0 \quad\text{for all}\quad \alpha \in R^+,\ k\ge \ell(\xi(\alpha)).
\end{equation}

\subsection{More About Fusion Products}
Given a filtered $\lie g[t]$-module $V$, recall that $F^rV$ denotes the corresponding filtered piece. Quite clearly
\begin{equation}\label{e:tsas}
\left(x\otimes (t^s - f(t))\right)w \in F^{r+s-1}V \quad\text{for all}\quad x\in\lie g,\ r,s\in\mathbb Z,\ w\in F^rV,
\end{equation}
and monic polynomial $f$ of degree $s$.

\begin{lem}\label{l:fusiontrunc}
	Let $l>0$ and, for each $1\le j\le l$, let $V^j$ be a finite-dimensional cyclic graded $\lie g[t]$-module generated by a vector $v_j$. Let $x\in\lie g$, suppose $N_j\in\mathbb Z_{\ge 0}, 1\le j\le l$, satisfy $(x\otimes t^{N_j})v_j=0$, and set $\lambda=\lambda_1+\cdots+\lambda_l, N=N_1+\cdots+N_l$. Then, for any choice  of distinct $a_1,\dots,a_l\in\mathbb C$, the vector $v_1*\cdots*v_l$ of $V_{a_1}^1 * \cdots * V_{a_l}^l$ satisfies $(x\otimes t^N)v_1*\cdots*v_l=0$.
\end{lem}

\begin{proof}
	Let $f(t) = \displaystyle \prod_{j=1}^{l} (t-a_j)^{N_j}$ and $v=v_1*\cdots*v_l$. By \eqref{e:tsas}, we have 
	\begin{equation*}
	(x\otimes t^N)v = (x\otimes f(t))v.
	\end{equation*}
	On the other hand, in $V^1_{a_1}\otimes\cdots\otimes V^l_{a_l}$, we have
	\begin{equation*}
	(x\otimes f(t))(v_1\otimes \cdots\otimes v_l)  = \sum_{j=1}^l v_1 \otimes\cdots\otimes  \left(x\otimes f(t+a_j)\right)v_j \otimes\cdots\otimes v_l = 0.
	\end{equation*}
\end{proof}

\begin{prop}\label{fusiontrunc}
	Let $l>0$ and, for each $1\le j\le l$, let $V^j$ be a quotient of $W(\lambda_j)$ for some $\lambda_j\in P^+$ and $v_j\in V^j_{\lambda_j}\setminus\{0\}$. Suppose $N_j\in\mathbb Z_{\ge 0}, 1\le j\le l$, satisfy 
	\begin{equation}\label{e:fusiontrunc}
	(\lie n^-\otimes t^{N_j}\mathbb C[t])v_j=0
	\end{equation}
	and set $\lambda=\lambda_1+\cdots+\lambda_l$ and $N=N_1+\cdots+N_l$. Then, for any  choice of distinct $a_1,\dots,a_l\in\mathbb C$, there exists an epimorphism of graded $\lie g[t]$-modules $W_N(\lambda)\to V_{a_1}^1 * \cdots * V_{a_l}^l$.
\end{prop}

\begin{proof}
	Denote by $v$ the image of $v_1\otimes \cdots\otimes v_l$ in $V_{a_1} (\lambda _1) * \ldots * V_{a_l}(\lambda_l)$. Then, quite clearly, $n^+[t]v=\lie h[t]_+v=0, h(v)=\lambda(h) v$ for all $h\in\lie h$, and $(x_i^-)^{\lambda(h_i)+1}v=0$ for all $i\in I$. Thus, by \eqref{e:relfor WN}, it suffices to show that $(x_\alpha^-\otimes t^N)v = 0$ for all $\alpha\in R^+$. But this follows from \eqref{e:fusiontrunc} and Lemma \ref{l:fusiontrunc}.
\end{proof}

Recall the definition \eqref{e:fusKR} and, given $i\in I, N>0$, and $\bsr i\in I^N$, set $S_i(\bsr i) = \{j: i_j=i\}$. 

\begin{thm}[{\cite[Theorem B]{naoi:tpkr}}]\label{t:naoifkr}
	For every $N>0$, $\bsr i\in I^N$, and $\bsr m\in \mathbb Z_{\ge 0}^N$, the module $KR_{\bsr i}(\bsr m)$ is isomorphic to the quotient of $W(\lambda)$ by the submodule generated by 
	\begin{equation*}%\label{e:naoirelfKR}
	x_{\alpha_i}^-(r,s)\, v \quad\text{for all}\quad i\in I,\ r>0,\ s+r>\sum_{j\in S_i(\bsr i)}\min\{r,m_j\}.	   
	\end{equation*}\hfill\qedsymbol
\end{thm}

\subsection{Proof of Theorem \ref{t:cvtruncated}}\label{ss:cvtruncated}
Let $\xi = \xi^\lambda_N$ and $v_\xi\in CV(\xi_N^\lambda)_\lambda\setminus\{0\}$. It follows from \eqref{e:cv2ndrelparts} that
$$ (x^-_\alpha \otimes t^N)v_\xi = 0 \quad\text{for all}\quad \alpha \in R^+,$$
and, hence, there exists a surjective homomorphism of $\lie g[t]$-modules 
\begin{eqnarray}\label{3}
W_N(\lambda ) \twoheadrightarrow CV(\xi_N^\lambda ).
\end{eqnarray}
To prove the converse, observe that Proposition \ref{p:cv2ndrel} implies that it suffices to show that
\begin{eqnarray*}
	_kx^-_\alpha(r,s) w=0 \quad\text{for all}\quad \alpha \in R_+, s, r, k \in \mathbb Z_{>0} \quad\text{s.t.}\quad s+r > rk + \sum _{j>k} \xi_N^\lambda(\alpha)_j, 
\end{eqnarray*}	
with $w\in W_N(\lambda)_\lambda$. Thus, by considering the subalgebra $\lie{sl}_\alpha[t]$, it suffices to prove  Theorem \ref{t:cvtruncated} for $\lie g=\lie{sl}_2$, which was done in \cite{kl:tor}, as mentioned earlier (see also \cite{mar:PhD} for a slightly modified proof).

\subsection{Proof of Theorem \ref{t:min}}

Given $\bsr i$ and $\bsr m$ as in Theorem \ref{t:naoifkr}, recall the definition of $S_i(\bsr i), i\in I$, just before Theorem \ref{t:naoifkr}, set 
\begin{equation*}
\bsr m_i = (m_j)_{j\in S_i(\bsr i)} \quad\text{for all}\quad i\in I
\end{equation*}
and let $\xi_{\bsr i}^{\bsr m}\in\mathscr P_{\lambda}$ be given by
\begin{equation*}
\xi_{\bsr i}^{\bsr m}(\alpha) = \begin{cases} \bsru m_i, & \text{if }\alpha = \alpha_i \text{ for some } i\in I,\\ (1^{\lambda(h_\alpha)}), &\text{otherwise.} \end{cases}
\end{equation*}
The following corollary was observed in \cite[Remark 3.4(b)]{naoi:tpkr} and a proof can be found in \cite[Corollary 2.4.1]{mar:PhD}.

\begin{cor}
	For every $N>0$, $\bsr i\in I^N$, and $\bsr m\in \mathbb Z_{\ge 0}^N$, there is an isomorphism $CV(\xi_{\bsr i}^{\bsr m})\cong KR_{\bsr i}(\bsr m)$.\hfill\qedsymbol
\end{cor}

\begin{proof}[Proof of Corollary \ref{c:KR>>W}]
	Let $\bsr m = \xi^{m\omega_i}_N(\alpha_i), m,N>0, i\in I$.
	Since $W_N(m\omega_i)\cong CV(\xi^{m\omega_i}_N)$ by Theorem \ref{t:cvtruncated} and $KR_{i}(\bsr m)\cong CV(\xi_{\bsr i}^{\bsr m})$ with $\bsr i=(i^{(N)})$ by the previous corollary, we are left to show that there exists an epimorphism
	\begin{equation*}
	CV(\xi_{\bsr i}^{\bsr m}) \to CV(\xi^{m\omega_i}_N).
	\end{equation*}
	Letting $v\in CV(\xi^{m\omega_i}_N)_\lambda$, this is in turn equivalent to showing that 
	\begin{equation*}
	 x_\alpha^-(r,s)\, v = 0 \quad\text{for all}\quad \alpha\in R^+,\ r,s>0 \quad\text{s.t.}\quad s+r \geq 1+ rk + \sum_{j>k} \xi_{\bsr i}^{\bsr m}(\alpha)_j
	\end{equation*}
	for some $k>0$. Since $\xi_{\bsr i}^{\bsr m}(\alpha_i) = \xi^{m\omega_i}_N(\alpha_i)$ by definition, we can assume $\alpha$ is not simple, in which case $r \geq 1=\xi (\alpha)_1$ and we are done  by Lemma \ref{l:rav}.
\end{proof}

It follows from \eqref{e:skr} that all maps in \eqref{e:theprojs} are isomorphisms if $\het_i(\theta)=1$. Write $\xi^{m\omega_i}_N(\alpha_i) = (m_1,\dots,m_l)$ and let $$\bs{\lambda} = (m_1\omega_i,\dots,m_l\omega_l).$$
Note that $l=N$ if $N\le m$ while $l=m$ and $m_j=1$ for all $1\le j\le m$, otherwise. 
In order to prove Theorem \ref{t:min}, it remains to observe that $\bs{\lambda}$ is a maximal element of $P^+(m\omega_i,N)$. By \cite[Lemma 3.3]{cfs:schur}, an element $\bs{\mu} = (\mu_1,\dots,\mu_k)\in P^+(m\omega_i,N)$ is maximal if and only if  
\begin{equation}
 \max_{1\le j\le k} \mu_j(h_i) - \min_{1\le j\le k} \mu_j(h_i) \le 1.
\end{equation}
Since $(m_1 - m_l) \leq 1$ by definition of $\xi^{m\omega_i}_N(\alpha_i)$, it follows that $\bs{\lambda}$ is indeed maximal.

\subsection{Proof of Proposition \ref{p:W=T}}
Let $i\in I, m,N>0$, $\bsr m = \xi^{m\omega_i}_N(\alpha_i)$, $v\in KR_i(\bsr m)_{m\omega_i}$, and $u$ be the image of $v$ in $T_i(\bsr m)$. 
By the existence of the first epimorphism in \eqref{e:theprojs}, we are left to show that there exists an epimorphism
\begin{equation*}
 W_N(m\omega_i)\to T_i(\bsr m).
\end{equation*} 
In light of \eqref{e:relfor WN} and the universal property of truncated Weyl modules mentioned after \eqref{e:defgrweyl}, to prove this, it suffices to show that 
\begin{equation}\label{e:W>T}
  (x_\alpha^-\otimes t^N)u=0 \qquad\text{for all}\qquad \alpha\in R^+.
\end{equation}
Evidently,
\begin{equation}
 \het_i(\alpha)=0  \qquad\Rightarrow\qquad x_\alpha^-v=0
\end{equation}
which implies \eqref{e:W>T} for such roots. Next, let us show that 
\begin{equation}\label{e:W>T1}
(x_\alpha^-\otimes t^N)v=0 \qquad\text{for all}\qquad \alpha\in R^+_{i,1},
\end{equation}
which implies \eqref{e:W>T} for such roots.

	Let $\ell(\bsr m)=l$ and $m_1\ge \cdots\cdots m_l$ be its nonzero parts. Recall once more that $l=N$ if $N\le m$ and $l=m$ and $m_j=1$ for all $1\le j\le m$ otherwise. In any case, $l\le N$. Without loss of generality assume $v=v_1*\cdots*v_l$ with $v_j\in KR(m_j\omega_i)_{m_j\omega_i}\setminus\{0\}$. In particular, by definition, 
	\begin{equation}\label{e:W=T}
	 (x_i^-\otimes t)v_j =0 \qquad\text{for all}\qquad 1\le j\le l.
	\end{equation}
We proceed by induction on $\het(\alpha)$. If $\het(\alpha)=1$, then $\alpha=\alpha_i$ and \eqref{e:W=T}, together with Lemma \ref{l:fusiontrunc}, implies that $(x_i^-\otimes t^l)v=0$ and, hence, $(x_i^-\otimes t^N)v=0$ since $l\le N$. If $\het(\alpha)>1$, we can write $\alpha=\beta+\gamma$ with $\beta,\gamma\in R^+$ and we may assume, without loss of generality, that $\het_i(\beta)=0$.
	It follows that $x_\beta^-v=0$ and, by induction hypothesis, $(x_\gamma^-\otimes t^l)v=0$. Hence, $[x_\beta^-,x_\gamma^-\otimes t^l]v=0$ and \eqref{e:W>T1} follows. 
Finally, assume $\alpha\in R^+_{i,k}, k\ge 2$. By definition \eqref{e:Tidef}, \eqref{e:W>T} holds with $\alpha = \beta_k:=\min R^+_{i,k}, k\ge 2$. If  $\alpha\in R^+_{i,k}\setminus\{\beta_k\}$, there exist $m\ge 1$ and $\gamma_j\in R^+_{i,0}, 1\le j\le m$ such that
\begin{equation}
  \beta_k + \sum_{j=1}^{s}\gamma_j\in R^+ \quad\text{for all}\quad 1\le s\le m \qquad\text{and}\qquad \alpha= \beta_k + \sum_{j=1}^{m}\gamma_j
\end{equation}
Since $x_{\gamma_j}^-v=0$ for all $1\le j\le m$,
\begin{equation*}
  [x_{\gamma_m}^-,\cdots[x_{\gamma_1}^-,x_{\beta_k}^-\otimes t^N]\cdots] u =0,
\end{equation*}  
which completes the proof of Proposition \ref{p:W=T}.

\begin{ex}
	We give the easiest example in support of Conjecture \ref{cj:new}. Namely, let $\lie g$ be of type $B_2$ and assume $I=\{1,2\}$ is such that $\alpha_2$ is the short simple root. Consider the case $\bsr m = (m,1)$ for some $m\ge 1$. In this case, Conjecture \ref{cj:new} states that 
	\begin{equation}\label{e:b2cj}
	  V(m\omega_2)* V(\omega_2)  \cong  (KR(m\omega_2)* KR(\omega_2)) / M
	\end{equation}
	where $M$ is the submodule of $KR(m\omega_2)* KR(\omega_2)$ generated by $(x_{\theta}^-\otimes t^2)v$ with $v$ a highest-weight generator of $KR_2(\bsr m)$.
	It follows from \cite{cha:fer,cm:kr} that, if $v_k\in KR(k\omega_2)_{k\omega_2}\setminus\{0\}$, then 
	\begin{equation}\label{e:krroots}
	   KR(k\omega_2) = \sum_{r=0}^{\lfloor k/2\rfloor} U(\lie n^-)(x_{\theta}^-\otimes t)^r v_k, \qquad (x_{\theta}^-\otimes t^2) v_k=0,
	\end{equation}
	and $(x_{\alpha}^-\otimes t) v_k=0$ if $\alpha\ne\theta$. In particular,
	\begin{equation}\label{e:krb}
	  KR(k\omega_2)[r]\cong V((k-2r)\omega_2) \qquad\text{for}\qquad 1\le r\le k/2
	\end{equation}
	and \eqref{e:b2cj} is immediate when $m=1$. 
	
	One can check that 
	\begin{equation}\label{e:tpb2}
	  V(k\omega_2)\otimes V(\omega_2) \cong V((k+1)\omega_2)\oplus V(\omega_1+(k-1)\omega_2)\oplus V((k-1)\omega_2)
	\end{equation}
	for all $k\ge 1$. Since $V_2(\bsr m)$ is a quotient of $W_2((m+1)\omega_2)$ isomorphic to $V(m\omega_2)\otimes V(\omega_2)$ as a $\lie g$-module, proceeding as in the proof of \eqref{e:krroots}, it follows that
	\begin{equation}
	  V_2(\bsr m) = U(\lie n^-)(x_{\alpha_2}^-\otimes t)v' \oplus U(\lie n^-)(x_{\theta}^-\otimes t)v'
	\end{equation}
	where $v'\in V_2(\bsr m)_{(m+1)\omega_2}\setminus\{0\}$. Moreover, one can check that this implies that $V_2(\bsr m)$ is the quotient of $W_2((m+1)\omega_2)$ by the submodule generated by
	\begin{equation}
	  (x_{\alpha_2}^-\otimes t)^2w, \quad (x_{\theta}^-\otimes t)^2w, \quad\text{and}\quad (x_{\alpha_2}^-\otimes t)(x_{\theta}^-\otimes t)w,
	\end{equation}
	where $w\in W_2((m+1)\omega_2)_{(m+1)\omega_2}\setminus\{0\}$. Letting $u$ be the image of $v$ in $T_2(\bsr m)$, \eqref{e:b2cj} follows if one shows that
	\begin{equation}\label{e:b2cj1}
	(x_{\alpha_2}^-\otimes t)^2u =(x_{\theta}^-\otimes t)^2u= (x_{\alpha_2}^-\otimes t)(x_{\theta}^-\otimes t)u=0.
	\end{equation}
	The relation $	(x_{\alpha_2}^-\otimes t)^2u =0$ follows from Theorem \ref{t:naoifkr} for all $m$. 
	Let us check the others for $m=2$. It follows from \eqref{e:krb} and  \eqref{e:tpb2} that 
	\begin{equation*}
	  KR_2(\bsr m) \cong_\lie g V(3\omega_2)\oplus V(\omega_1+\omega_2)\oplus V(\omega_2)^{\oplus 2}.
	\end{equation*}
	One can then proceed as in the proof of \eqref{e:krroots} to show that
	\begin{equation*}
	  KR_2(\bsr m)= U(\lie n^-)v \oplus U(\lie n^-)(x^-_{\alpha_2}\otimes t)v \oplus U(\lie n^-)(x^-_{\theta}\otimes t)v \oplus U(\lie n^-)(x^-_{\theta}\otimes t^2)v
	\end{equation*}
	from where 	\eqref{e:b2cj1} follows immediately.
			\endd
\end{ex}

\section{Further Results for Rank One}\label{s:sl2}

\subsection{Demazure Flags}\label{ss:flags}

Given $\ell\in\mathbb Z_{\ge 0}$ and $\lambda\in P^+$, the $\lie g$-stable level-$\ell$ Demazure module $D(\ell,\lambda)$ is the quotient of $W(\lambda)$ by the submodule generated by
\begin{equation}\label{e:D(ell,lambda)}
\{(x_\alpha^-\otimes t^{s_\alpha}) v: \alpha\in R^+\}\cup \{(x_\alpha^-\otimes t^{s_\alpha-1})^{m_\alpha+1} v: \alpha\in R^+ \text{ s.t. } m_\alpha<\ell r^\vee_\alpha\},
\end{equation}
where $v$ is a highest-weight generator of $W(\lambda)$ and $r^\vee_\alpha, s_\alpha$, and $m_\alpha$ are the integers defined by
\begin{equation*}
r^\vee_\alpha = \begin{cases}
1,& \text{ if } \alpha \text{ is long,}\\
r^\vee,& \text{ if } \alpha \text{ is short},
\end{cases}
\quad\text{and}\quad \lambda(h_\alpha) = (s_\alpha-1)\ell r^\vee_\alpha + m_\alpha, \quad 0<m_\alpha\le \ell r^\vee_\alpha,
\end{equation*}
where $r^\vee$ is the lacing number of $\lie g$. In particular, if $\lie g$ is simply laced, it follows from \eqref{e:endofweyl} that 
\begin{equation}\label{e:weyl=dem}
W(\lambda)\cong D(1,\lambda).
\end{equation}
It is well known that there are epimorphisms of graded $\lie g[t]$-modules
\begin{equation}\label{e:projdems}
D(\ell,\lambda)\twoheadrightarrow D(\ell',\lambda) \qquad\text{for all}\qquad \lambda\in P^+,\ \ell\le\ell'.
\end{equation}
In particular, $D(\ell,\lambda)\cong V_0(\lambda)$ if $\ell$ is sufficiently large.

Let $\xi_{\ell,\lambda}\in\mathscr P_\lambda$ be given by 
\begin{equation}
\xi_{\ell,\lambda}(\alpha) = ((\ell r^\vee_\alpha)^{(s_\alpha-1)},m_\alpha) \quad\text{for}\quad \alpha\in R^+.
\end{equation}
It was shown in \cite{cv:cvm} that we have an isomorphism of graded $\lie g[t]$-modules:
\begin{equation}\label{e:Dem=CV}
D(\ell,\lambda)\cong CV(\xi_{\ell,\lambda}).
\end{equation}
Set
\begin{equation*}
D(\ell,\lambda,m) = \tau_m D(\ell,\lambda).
\end{equation*}
A $\lie g[t]$-module $V$ admits a Demazure flag of level-$\ell$ if there exist $k>0, \lambda_j\in P^+, m_j\in\mathbb Z, j=1,\dots,k$, and a sequence of inclusions
\begin{equation}\label{e:demflag}
0 = V_0 \subset V_1 \subset \cdots \subset V_{k-1} \subset V_k = V \quad\text{with}\quad V_j/V_{j-1}\cong D(\ell, \lambda_j,m_j) \ \forall\ 1\le j\le l.
\end{equation}
Let $\mathbb V$ be a Demazure flag of $V$ as in \eqref{e:demflag} and, for a Demazure module $D$, define the multiplicity of $D$ in $\mathbb V$ by
\begin{equation*}
[\mathbb V:D] = \#\{1\le j\le l : V_j/V_{j-1}\cong D\}.
\end{equation*}
As observed in \cite[Lemma 2.1]{cssw:demflag}, the multiplicity does not depend on the choice of the flag for fixed $\ell$ and, hence, by abuse of language, we shift the notation from $[\mathbb V:D]$ to $[V:D]$. Also following \cite{cssw:demflag}, we consider the generating polynomial
\begin{equation*}
[V:D](t) = \sum_{m\in\mathbb Z}\ [V:\tau_mD]\ t^m \ \in\ \mathbb Z[t,t^{-1}].
\end{equation*}
In the category of non-graded $\lie g[t]$-modules we have $\tau_mD\cong D$ and, hence, one can also be interested in computing the ungraded multiplicity of $D$ in $V$ which is given by 
\begin{equation*}
[V:D](1) = \sum_{m\in\mathbb Z}\ [V:\tau_mD].
\end{equation*}

\subsection{Demazure Flags for Truncated Weyl Modules}
For the remainder of the paper, let $\lie g=\lie{sl}_2$, $q=q_{\alpha_1}$ and $p=p_{\alpha_1}$ as defined in \eqref{e:eucdiv} and identify $\lambda\in P^+$ with $\lambda(h_1)$. Also, given $\xi\in\mathscr P_\lambda$, we write $\xi_j$ instead of $\xi(\alpha_1)_j$ for simplicity.  

Theorem \ref{t:cvtruncated} together with \eqref{e:Dem=CV} implies that
\begin{equation*}
	W_N(\lambda)\cong \begin{cases}
	D(q,\lambda),& \text{if } N|\lambda,\\
	D(q+1,\lambda),& \text{if } p\in \{N-1,\lambda\}. 
	\end{cases}
\end{equation*}
Note that $p=\lambda$ if and only if $N>\lambda$. We will see below that  $W_N(\lambda)$ is not a Demazure module for all other values of $p$. 
Observe that, if $N=2$, then $p\in\{0,1\}=\{0,N-1\}$ and, hence, $W_N(\lambda)$ is always a Demazure module.
The following was proved in \cite{cssw:demflag}.
\begin{thm}\label{t:CVDemflag}
let $\xi\in\mathscr P_\lambda$ for some $\lambda\in P^+$. Then, $CV(\xi)$ admits a level-$\ell$ Demazure flag if and only if $\ell\ge \xi_1$. In particular, $D(\ell,\lambda)$  admits a level-$\ell'$ Demazure flag iff $\ell'\ge\ell$.\hfill\qedsymbol
\end{thm}
Together with \eqref{e:weyl=dem}, this theorem implies that $W(\lambda)$ admits a level-$\ell$ Demazure flag for all $\ell\ge 1$. 
We have the following corollary of Theorems \ref{t:cvtruncated} and \ref{t:CVDemflag}.

\begin{cor}
	The module $W_N(\lambda)$  admits a level-$\ell$ Demazure flag if and only if 
	\begin{equation*}
	\ell\ge \begin{cases}
	q, &\text{if } N|\lambda,\\
	q+1, &\text{otherwise.} 
	\end{cases}
	\end{equation*}
	\hfill\qedsymbol
\end{cor}

In light of \eqref{e:projdems}, it follows that, in order to show that $W_N(\lambda)$ is not a Demazure module for $p\ne 0,N-1$, it suffices to show that its level-$(q+1)$ Demazure flag has length bigger than 1. To see this, we will use the main tool of the proof of Theorem \ref{t:CVDemflag}. Namely, it was shown in  \cite{cv:cvm} that, if $\xi\in\mathscr P_\lambda$ has at least two nonzero parts,  there exists a short exact sequence of $\lie g[t]$-modules
\begin{equation}\label{e:CVshortseq}
0 \rightarrow \tau _{(l-1)\xi_l} CV(\xi^-) \rightarrow CV(\xi) \rightarrow CV(\xi^+) \rightarrow 0,
\end{equation}
where $l$ is the number of nonzero parts of $\xi$, $\xi^- =  (\xi^- _1 \geq \ldots \geq \xi^-_{l-2} \geq \xi^-_{l-1} \geq 0)$ is given by
\begin{equation*}
\xi^-_j = \xi_j \quad\text{if}\quad j< l-1, \qquad \xi_{l-1}^- = \xi_{l-1} - \xi_l,
\end{equation*}
and $\xi ^+ =  (\xi^+_1 \geq \ldots \geq \xi^+_{l-1} \geq \xi^+ _{l} \geq 0)$ is the unique partition associated to the $l$-tuple 
$$(\xi_1, \ldots , \xi_{l-2},\xi_{l-1}+1, \xi_l -1).$$ Note that
\begin{equation}
\xi ^+ \in \mathscr P_\lambda \qquad\text{and}\qquad \xi^- \in \mathscr P_{\lambda -2\xi_l}.
\end{equation}
One also easily checks that
\begin{equation}
\xi=\xi_N^\lambda \quad\text{and}\quad p\ne 0,N-1 \quad\Rightarrow\quad \xi_1^\pm=q+1.
\end{equation}
Hence, the length of a level-$(q+1)$ Demazure flag of $CV(\xi)$ is the sum of the lengths of level-$(q+1)$ Demazure flags of $CV(\xi^\pm)$, showing that $W_N(\lambda)$ is not a Demazure module. In the examples, $\soc(M)$ denotes the socle of a module $M$.

\begin{ex}\label{ex:p=N-2}
If $p=N-2\ne\lambda$ we have a length-$2$ flag:
\begin{gather*}
0 \rightarrow D(q+1,\lambda-2q,(N-1)q) \rightarrow W_N(\lambda) \rightarrow D(q+1,\lambda) \rightarrow 0.
\end{gather*}
Consider the case that $\lambda=4$ and $N=3$, so $p=q=1$ and the above sequence becomes
\begin{gather*}
0 \rightarrow V(2,2) \rightarrow W_3(4) \rightarrow D(2,4) \rightarrow 0.
\end{gather*}
One can easily check using \eqref{e:CVshortseq} that we have exact sequences
\begin{gather*}
0 \rightarrow V(0,2) \rightarrow D(2,4)\rightarrow D(3,4) \rightarrow 0 \qquad\text{and}\qquad
0 \rightarrow V(2,1) \rightarrow D(3,4) \rightarrow V(4,0) \rightarrow 0.
\end{gather*}
This implies that
$\soc(D(2,4))\cong V(0,2)$ and
\begin{gather*}
\soc(W_3(4))= W_3(4)[2]\cong V(2,2) \oplus V(0,2).
\end{gather*}
This shows that, differently from the non truncated case, truncated Weyl modules may
have non simple socle.\endd
\end{ex}

\begin{ex}
If $p=N-3\ne\lambda$, the flag has length $2$ or $3$. To see this, observe that
\begin{equation*}
\xi^+ = ((q+1)^{N-2},q,q-1) \qquad\text{and}\qquad \xi^- = ((q+1)^{N-3},q).
\end{equation*}
In particular, $CV(\xi^-)\cong D(q+1,\lambda-2q)$. If $q=1$ (i.e., $\lambda=2N-3$), we have a length-$2$ flag:
\begin{gather*}
0 \rightarrow D(2,\lambda-2,N-1) \rightarrow W_N(\lambda) \rightarrow D(2,\lambda) \rightarrow 0.
\end{gather*}
Otherwise, one easily checks using \eqref{e:CVshortseq} that $CV(\xi^+)$ has a lenght-$2$ flag:
\begin{gather*}
0 \rightarrow D(q+1,\lambda-2(q-1),(N-1)(q-1)) \rightarrow CV(\xi^+) \rightarrow D(q+1,\lambda) \rightarrow 0.
\end{gather*}

Consider the case $\lambda=5$ and $N=4$. Then, $p=q=1=N-3$ and we have the following exact sequence
\begin{gather*}
0 \rightarrow D(2,3,3) \rightarrow W_4(5) \rightarrow D(2,5) \rightarrow 0.
\end{gather*}
The grading series is described by
\begin{center}
	\begin{tabular}{ | c | c | c |}
		\hline
		\text{degree}& $D(2,5)$& $D(2,3,3)$\\
		\hline
		0 & $V(5)$ &  \\ \hline
		1 & $V(3)$ &  \\ \hline
		2 & $V(3)\oplus V(1)$ &  \\ \hline
		3 & $V(1)$ &  $V(3)$ \\ \hline
		4 &  & $V(1)$   \\
		\hline
	\end{tabular}
\end{center}
It follows that $\soc(W_4(4))$ is isomorphic either to  $V(1,4)$ or to $V(1,4)\oplus V(1,3)$. Let us show that it is former, i.e., $W_4(5)$ has simple socle. Indeed, the socle is not simple if and only if there exists nonzero $v\in W_4(5)[3]_1$ satisfying
\begin{equation*}
   (x^+\otimes t^r)v = (h\otimes t^r)v = 0 \qquad\text{for all}\qquad r\ge 0.
\end{equation*} 
To simplify notation, set
\begin{equation*}
  x^\pm_r = x^\pm\otimes t^r \qquad\text{and}\qquad h_r = h\otimes t^r.
\end{equation*}
Let $w\in W_4(5)_5$ be nonzero. The weight subspace $W_4(5)[3]_1$ is spanned by $x_0^-x_3^-w$ and $x_2^-x_1^-w$. One can then easily check that a linear combination of such vectors is killed by $x^+_r$ for all $r\ge 0$ if and only if it is a scalar multiple of 
\begin{equation*}
   (x_2^-x_1-x_0^-x_3^-)w.
\end{equation*}
Since
\begin{equation*}
  h_1(x_2^-x_1-x_0^-x_3^-)w = -2(x_2^-)^2w,
\end{equation*}
it suffices to check that $(x_2^-)^2w\ne 0$. The vectors $(x_2^-)^2w$ and $x^-_1x^-_3w$ clearly span $W_4(5)[4]_1$ which is not zero by the above table. On the other hand, Proposition \ref{p:cv2ndrel} implies $x_{\alpha_1}^-(2,4)w$$=0$, i.e.,  $(x_2^-)^2w=-2x^-_1x^-_3w$ and, hence, the assertion follows.
\endd
\end{ex}

The following characterization of the truncated Weyl modules having a Demazure flag of length $2$ is easily deduced from the computations of the above two examples.

\begin{prop}
	Suppose $p\ne 0,N-1$. The level-$(q+1)$ Demazure flag of $W_N(\lambda)$ has lenght $2$ if and only if either $p=N-2\ne\lambda$ or $p=N-3$ and $q=1$. \endd
\end{prop} 

\begin{rem}
	By \cite[Lemma 2.3]{kona}, since $W_N(\lambda)$ obviously has a simple head, its radical series coincide with its grading series. Example \ref{ex:p=N-2} shows that truncated Weyl modules may not have simple socle and, hence, \cite[Lemma 2.3]{kona} does not guarantee that the socle series coincides with the grading series. However, that is the case in Example, \ref{ex:p=N-2}.\endd
\end{rem}

Given $a,b,\ell\in\mathbb Z_{\ge 0}$, let 
\begin{equation*}
\xi_{a,b}^\ell = ((\ell+1)^{(a)},\ell^{(b)}), \qquad \lambda_{a,b}^\ell = \ell(a+b)+a,
\end{equation*}
and note that
\begin{equation}\label{e:abpq}
\xi_N^\lambda = \xi_{p,N-p}^q
\end{equation}
or, equivalently,
\begin{equation}
 CV(\xi_{a,b}^\ell)\cong W_{a+b}(\lambda_{a,b}^\ell).
\end{equation}
In particular,
\begin{equation}
 CV(\xi^{\ell-1}_{N,0})= CV(\xi^{\ell}_{0,N})\cong D(\ell,N\ell)\cong W_N(N\ell).
\end{equation}
Note also that, for $\lambda=\lambda_{a,b}^1$ and $N=a+b, b\ne 0$, \eqref{e:CVshortseq} can be rewritten as
\begin{equation}\label{e:inc1}
0 \rightarrow \tau_{N-1} W_{N-1-\delta_{p,N-1}}(\lambda-2) \rightarrow W_N(\lambda) \rightarrow W_{N-1}(\lambda) \rightarrow 0.
\end{equation}

Given $\mu\in P^+$, consider the function
\begin{equation}
\gamma_{a,b}^\ell(\mu,t) = [ CV(\xi_{a,b}^\ell):D(\ell+1,\mu)](t).
\end{equation}
Such functions were studied in \cite{bcsv:dfcpmt,cssw:demflag}, but a full understanding is still not achieved. For instance, it follows from \cite{cssw:demflag} that 
\begin{equation}\label{e:csswmf2}
\gamma^1_{0,\lambda}(\lambda-2k,t) = [W(\lambda):D(2,\lambda-2k)](t) =  t^{k\lceil \lambda/2\rceil} \tqbinom{\lfloor \lambda/2\rfloor}{k}_t 
\end{equation}
for all $0\le k\le\lfloor\lambda/2\rfloor$ where
\begin{equation}\label{e:qbin}
\qbinom{m}{k}_t = \prod_{j=0}^{k-1} \frac{1-t^{m-j}}{1-t^{k-j}} \quad\text{for}\quad 0\le k\le m.
\end{equation}
Henceforth, assume $a>0$. Note that, 
\begin{equation*}
  \xi^\ell_{a,b} = \xi_{\ell+1,\lambda^\ell_{a,b}} \qquad\text{if}\qquad b=0,1
\end{equation*}
and, therefore, it follows from \eqref{e:Dem=CV} that $CV(\xi^\ell_{a,b})\cong D(\ell+1,\lambda^\ell_{a,b})$. In particular,
\begin{equation}
\gamma_{a,b}^\ell(\mu,t)=\delta_{\lambda^\ell_{a,b},\mu} \qquad\text{if}\qquad b=0,1.
\end{equation}
Hence, we can assume $b\ge 2$. In that case, it follows from \eqref{e:inc1} that
\begin{equation}
	\gamma^1_{a,b}(\mu,t) = \gamma^1_{a-1,b+2}(\mu,t) -  \gamma^1_{a-1,b}(\mu,t)t^{a+b},  
\end{equation}
which, combined with \eqref{e:csswmf2}, gives a recursive procedure to compute $\gamma^1_{a,b}(\mu,t)$. However, one can use  another approach leading to a formula without minus signs, as we shall see next.

Given a partition $\xi$, let $\xi^*$ be the partition obtained from $\xi$ by removing its largest part. In particular, if $\xi \in \mathscr P_\lambda$, then $$\xi^* \in \mathscr P_{\lambda-\xi_1}.$$
Note that, for $a>0$, we have
\begin{equation}\label{e:xiab*}
(\xi_{a,b}^\ell)^* = \xi_{a-1,b}^\ell
\end{equation}
The following is \cite[Lemma 3.8]{cssw:demflag}
\begin{equation}\label{e:cssw38}
[CV(\xi):D(\xi_1, \mu)](t) = t^{\frac{\lambda - \mu}{2}}\ [CV(\xi^*):D(\xi_1, \mu - \xi_1)](t).
\end{equation}
In particular, 
\begin{equation*}
[CV(\xi):D(\xi_1, \mu)](t) = 0 \quad\text{if}\quad \xi_1>\mu.
\end{equation*}
Using \eqref{e:xiab*} and iterating \eqref{e:cssw38} we get
\begin{equation}\label{e:deltaiter}
\gamma_{a,b}^\ell(\mu,t)= t^{\frac{a}{2}(\lambda_{a,b}^\ell - \mu)}\ \gamma_{0,b}^\ell(\mu-a(\ell+1),t).
\end{equation}
In particular, \eqref{e:abpq} implies
\begin{equation}\label{e:truncDq}
[W_N(\lambda):D(q+1, \mu)](t) = t^{\frac{p}{2}(\lambda - \mu)}\ [D(q,q(N-p)):D(q+1,\mu-p(q+1))](t).
\end{equation}

\begin{cor}\label{c:mult}
	Let $\lambda\in P^+$ and $N>1$ be such that $N\le\lambda< 2N$. For all $0\le k\le\lambda/2$, we have
	\begin{equation*}
	[W_N(\lambda):D(2,\lambda-2k)](t)= \begin{cases}
	t^{k\lceil \lambda/2\rceil} \tqbinom{N-\lceil \lambda/2\rceil}{k}_t, &\text{if } \, \,  k \le N-\lceil \lambda/2\rceil,\\
	0, &\text{otherwise.} 
	\end{cases}
	\end{equation*}
\end{cor} 

\begin{proof}
	Writing $\lambda=N+p$ with $0\le p<N$ and $\mu=\lambda-2k$, we have $$[W_N(\lambda):D(2,\mu)](t)=\gamma ^1_{p, N-p}(\mu,t).$$ 
	Hence, by \eqref{e:deltaiter},
	\begin{equation*}
	  \gamma^1_{p,N-p}(\mu,t) = t^{pk}\ \gamma^1_{0,N-p}(N-p - 2k,t) = t^{pk}\ \gamma^1_{0,\lambda-2p}(\lambda-2p - 2k,t).
	\end{equation*}
In particular, this is zero if $2k>\lambda-2p = 2N-\lambda$. Otherwise, we are done using \eqref{e:csswmf2}. 
\end{proof}

This corollary can be used to compute the length of the  level-$2$ Demazure flag. Namely, recall from \cite[Section 3.8]{cssw:demflag} that, for every partition $\xi$,
\begin{equation}
[CV(\xi):D(\ell,\mu)](t)\ne 0 \qquad\Rightarrow\qquad |\xi|-2\mu \in2\mathbb Z_{\ge 0}.
\end{equation}
Then, if we consider the generating function
\begin{equation*}
L_\xi^\ell(x,t) = \sum_{k=0}^{\lfloor|\xi|/2\rfloor} [CV(\xi):D(\ell,|\xi|-2k)](t)\ x^k,
\end{equation*}
the length of the level-$\ell$ Demazure flags of $CV(\xi)$ is
\begin{equation}
L_\xi^\ell := L_\xi^\ell(1,1).
\end{equation}
In the case that $\xi=\xi^\lambda_N$ with $\lambda$ and $N$ as in Corollary \ref{c:mult}, we get
\begin{equation}\label{e:Lq=1}
L_\xi^2 = \sum_{k=0}^{N-\lceil \lambda/2\rceil}  \binom{N-\lceil \lambda/2\rceil}{k} = 2^{N-\lceil \lambda/2\rceil}.
\end{equation}

Similar results can be obtained for $\gamma^2_{a,b}$, i.e., for $2N\le\lambda<3N$, by using \cite[Proposition 1.4]{bcsv:dfcpmt} in place of \eqref{e:csswmf2}. Since higher level analogues of \eqref{e:csswmf2} are still unavailable, no analogue of Corollary \ref{c:mult} and \eqref{e:csswmf2}  for $\ell N\le \lambda<(\ell+1)N$ with $\ell>2$ can be stated at this point.

\subsection{Final Remarks}\label{ss:inc}

Let $\lambda\in P^+ = \mathbb{Z}_{\geq 0}$ and $w\in W(\lambda)_\lambda\setminus\{0\}$. It is known from the identification with affine Demazure modules (see \cite{cl:wdf, foli:wdkr}) that the submodule generated by $W(\lambda)[\lambda-1]_{\lambda-2}$ or, equivalently, by the vector $(x^-\otimes t^{\lambda-1})w$, is isomorphic to $\tau_{\lambda-1}W(\lambda-2)$. In other words, we have an inclusion
\begin{equation*}
\tau_{\lambda-1} W(\lambda-2)\hookrightarrow W(\lambda).
\end{equation*}
Combining this with \eqref{e:relfor WN}, it follows that we have an exact sequence
\begin{equation}\label{e:incnt}
0\to \tau_{\lambda-1} W(\lambda-2)\to W(\lambda)\to W_{\lambda-1}(\lambda)\to 0.
\end{equation}
In other words, if $N= \lambda$ and $N'=N-1$, the kernel of the projection \eqref{e:truncproj} is, up to grade shift, isomorphic to $W(\lambda-2)\cong W_{N-2}(\lambda-2)$ and \eqref{e:incnt} can be rewritten as
\begin{equation}\label{e:incnt'}
0\to \tau_{\lambda-1} W_{\lambda-2}(\lambda-2)\to W_\lambda(\lambda)\to W_{\lambda-1}(\lambda)\to 0.
\end{equation}

Next, we make a few observations about two natural questions.  On one hand, it would be interesting to understand the kernel of \eqref{e:truncproj} with $N<\lambda$, specially for $N'=N-1$. Note that \eqref{e:inc1} gives the answer of this special case when $N<\lambda<2N$ and it coincides with the case $N=\lambda$ as seen in \eqref{e:incnt'}. Unfortunately, as we shall see in Example \ref{ex:kernotw}, the kernel is not always a truncated Weyl module and, in fact, may not even be a CV-module. On the other hand, in analogy to the non-truncated case, one would like to characterize all possible chains of  inclusions of truncated Weyl modules. For instance, if $\lambda = qN+p$ with $0\le p<N$ as before and, either $p<N-1$ or $q=1$, then \eqref{e:CVshortseq} gives rise to the inclusion
\begin{equation}%\label{e:incCV}
\tau _{(N-1)q} W_{N-2}(\lambda-2q) \hookrightarrow W_N(\lambda).
\end{equation}
The corresponding quotient is a truncated Weyl module if and only if $q=1$ which is \eqref{e:inc1} again. If $q>1$ and $p=N-1$, \eqref{e:CVshortseq} does not give rise to an inclusion of truncated Weyl modules, but a second application gives rise to the inclusion
\begin{equation}%\label{e:incCV}
\tau_{N-2}\tau _{(N-1)q} W_{N-2}(\lambda-2(q+1)) \hookrightarrow W_N(\lambda).
\end{equation}

Denote by $\pi_N^{N'}$ the projection \eqref{e:truncproj} and, for $N'=N-1$, simplify the notation and write $\pi_N$. Note 
\begin{equation*}
\pi_N^{N'} = \pi_{N'+1}\circ\cdots\circ\pi_{N-1}\circ\pi_N \qquad\text{for all}\qquad N'<N.
\end{equation*}
Moreover, \eqref{e:relfor WN} implies that 
\begin{equation*}
\ker(\pi_N) = U(\lie n^-[t])(x^-\otimes t^{N-1})w_N
\end{equation*}
where $w_N\in W_N(\lambda)_\lambda\setminus\{0\}$. One easily checks that we have a surjective map
\begin{equation*}
\varpi_N:\tau_{N-1}W_N(\lambda-2) \twoheadrightarrow \ker(\pi_N).
\end{equation*}
Let
\begin{equation*}
\delta_N(\lambda) = \dim(W_N(\lambda)).
\end{equation*}
The following is \cite[Theorem 5(ii)]{cv:cvm}.

\begin{thm}\label{t:cvfusion}
	Let $\lambda\in P^+$, $\xi\in\mathscr P_\lambda$, and $l=\ell(\xi)$. For any choice of distinct $a_1,\dots,a_l\in\mathbb C$, there exists an isomorphism  
	$$CV(\xi)\cong V_{a_1}(\xi_1) \ast \cdots \ast V_{a_l}(\xi_l)$$ of graded $\lie g[t]$-modules.\hfill\qedsymbol
\end{thm}

It follows from Theorems \ref{t:cvtruncated} and \ref{t:cvfusion} that
\begin{equation}
\delta_N(\lambda) = (q+2)^p(q+1)^{N-p}.
\end{equation}
Therefore, $\varpi_N$ is an isomorphism if and only if
\begin{equation}\label{e:kerNistw}
\delta_N(\lambda) -  \delta_N(\lambda-2) = \delta_{N-1}(\lambda).
\end{equation}
Noting that
\begin{equation*}
\xi_N^{\lambda-2} = 
\begin{cases}
((q+1)^{(p-2)}, q^{(N-p+2)}), &\text{if } p\ge 2,\\
(q^{(N-2+p)},(q-1)^{(2-p)} ), &\text{if } p=0,1,
\end{cases}
\end{equation*}
while
\begin{gather*}
  \xi_{N-1}^{\lambda} = ((q+q'+1)^{(p')}, (q+q')^{(N-1-p')})\\ \quad\text{with}\quad\\ p+q = q'(N-1)+p',\ 0\le p'<N-1,
\end{gather*}
one can rewrite \eqref{e:kerNistw} in terms of the parameters $q$ and $p$. In particular, one can easily check that \eqref{e:kerNistw} is always satisfied for $N=2$ and, hence, we have exact sequences
\begin{equation}
0\to \tau_1 W_2(\lambda-2)\to W_2(\lambda)\to W_1(\lambda)\cong V(\lambda)\to 0.
\end{equation}
Henceforth, assume $N>2$. For $q=1$, since 
$$N-1-\delta_{p,N-1}\ge \lambda-2 \qquad\Leftrightarrow\qquad p=0,1,$$
it follows \eqref{e:inc1} that $\varpi_N$ is injective if and only if $\lambda =N,N+1$. Hence, we may assume $q>1$.

\begin{ex}\label{ex:kernotw}
	The smallest example of non injective $\varpi_N$ happens with $N=3$ and $\lambda=6$. One can easily check that \eqref{e:kerNistw} is not satisfied. Alternatively, note that $W_3(6)\cong D(2,6)$ has simple socle. If $\varpi_3$ were injective, then it would contain a submodule isomorphic to $\tau_2W_3(4)$ which does not have simple socle by Example \ref{ex:p=N-2}. In fact, in this case, we see that there exists a short exact sequence
	\begin{equation*}
	0 \to V(0,2) \to W_3(4)\stackrel{\varpi_3}{\longrightarrow}\ker(\pi_3)\to 0
	\end{equation*}
	since
	\begin{equation*}
	\delta_3(6) - \delta_2(6) = 11 \qquad\text{and}\qquad \delta_3(4)=12.
	\end{equation*}
	It is now easy to see that  $\ker(\pi_3)$ is not a CV-module in this case. Indeed, if it were, then $\ker(\pi_3)$ would be isomorphic to $CV(\xi)$ with $\xi$ being a partition of $4$. However, using Theorem \ref{t:cvfusion}, one easily sees that $\dim(CV(\xi))\ne 11$ for all such partitions. 
	
	The only inclusion of a truncated Weyl module in $W_3(6)$ comes from \eqref{e:CVshortseq} which reads
	\begin{equation*}
	0\to\tau_2 W_1(2)\to W_3(6)\to CV(\xi)\to 0 \qquad\text{with}\qquad \xi = (3,2,1).
	\end{equation*}
	In particular, $\soc(W_3(6))\cong V(2,2)\cong \tau_2W_1(2)$.
	\endd
\end{ex}

\bibliographystyle{amsplain}

\end{document}